\documentclass[a4paper,reqno]{amsart}%
\usepackage{pgf,tikz}
\usepackage{amssymb}
\usepackage{amsmath}
\usepackage{amsfonts}
\usepackage{graphicx}
\usepackage{bm}%
\providecommand{\U}[1]{\protect\rule{.1in}{.1in}}
\usetikzlibrary{backgrounds}
\usetikzlibrary{arrows}
\let\oldmathbf\mathbf
\renewcommand{\mathbf}[1]{\boldsymbol{\oldmathbf{#1}}}
\newtheorem{theorem}{Theorem}

\newtheorem{conjecture}[theorem]{Conjecture}

\allowdisplaybreaks

\begin{document}

\title{Convergence of multiple Fourier series and Pick's theorem}
\author[L. Brandolini]{Luca Brandolini}
\address{Dipartimento di Ingegneria Gestionale, dell'Informazione e della Produzione,
Universit\`a degli Studi di Bergamo, Viale Marconi 5, Dalmine BG, Italy}
\email{luca.brandolini@unibg.it}
\author[L. Colzani]{Leonardo Colzani}
\address{Dipartimento di Matematica e Applicazioni, Universit\`a di Milano-Bicocca, Via
Cozzi 55, Milano, Italy}
\email{leonardo.colzani@unimib.it}
\author[S. Robins]{Sinai Robins}
\address{Departamento de ciência da computação,
Instituto de Matemática e Estatistica,
Universidade de São Paulo,
Brasil}
\email{sinai.robins@gmail.com}
\author[G. Travaglini]{Giancarlo Travaglini}
\address{Dipartimento di Matematica e Applicazioni, Universit\`a di Milano-Bicocca, Via
Cozzi 55, Milano, Italy}
\email{giancarlo.travaglini@unimib.it}

\date{}
\maketitle

\begin{abstract}
We add another brick to the large building comprising proofs of Pick's
theorem. Although our proof is not the most elementary, it is short and
reveals a connection between Pick's theorem and the pointwise convergence of
multiple Fourier series of piecewise smooth functions.

\end{abstract}

\bigskip
\bigskip
Let $P$ be an integer polygon, that is a simple polygon having integer
vertices in the Cartesian plane. Let $\left\vert P\right\vert $ be its area,
$I$ the number of integer points strictly inside $P$, and $B$ the number of
integer points on the boundary $\partial P$. Then

\begin{theorem}
[Pick]%
\begin{equation}
\left\vert P\right\vert =I+\frac{1}{2}B-1\ . \label{stat}%
\end{equation}

\end{theorem}

In spite of the simple statement, this is not a very old result. It was
published by Georg Pick in 1899, and first popularized by Hugo Steinhaus in
1937 in the Polish edition of \textit{Mathematical Snapshots}.

The theorem has many proofs and interesting features. Its statement can be
explained to primary school children, who could be asked to verify it on examples.

A relatively simple and well-known proof can be sketched as follows.

Step 1. A simple integer polygon can be triangulated into integer primitive
triangles, that is with no integer points other than the vertices.

Step 2. Both terms $\left\vert P\right\vert $ and $I+\frac{1}{2}B-1$ in
(\ref{stat}) are \textquotedblleft additive\textquotedblright\ with respect to
the above triangulation.

Step 3. A primitive triangle together with one of its reflections gives a
parallelogram which tiles the plane under integer translations.

Step 4. This latter parallelogram has area $1$, so that (\ref{stat}) holds
true for primitive triangles.

Pick's theorem can be related to certain non completely elementary topics in
mathematics. See e.g. \cite{Fun} for a connection to Euler's formula for
planar graphs, or \cite{RT} for a connection to Minkowski's theorem on integer
points in convex bodies, or \cite{DR} for a complex analytic proof.

The purpose of this paper is to exhibit a direct connection between Pick's
theorem and harmonic analysis. The Fourier analytic proof we give here is a
consequence of classical results on pointwise convergence of multiple Fourier
series. This proof is short, self contained and does not rely on any of the
previous geometrical steps. Moreover it suggests a point of departure for
higher dimensional investigations.

In what follows if $f\left(  x\right)  $ and $\varphi\left(  x\right)  $ are
integrable functions on $\mathbb{R}^{d}$ and if we define $\varphi
_{\varepsilon}\left(  x\right)  =\varepsilon^{-d}\varphi\left(  \varepsilon
^{-1}x\right)  $, then
\[
\varphi_{\varepsilon}\ast f\left(  x\right)  =\int_{\mathbb{R}^{d}}%
\varphi_{\varepsilon}\left(  x-y\right)  f\left(  y\right)  dy
\]
denotes the convolution. Moreover
\[
\widehat{f}\left(  \xi\right)  =\int_{\mathbb{R}^{d}}f\left(  x\right)
e^{-2\pi i\xi\cdot x}\ dx
\]
denotes the Fourier transform. It is well known that if $f\left(  x\right)  $
is integrable on $\mathbb{R}^{d}$, then $\sum_{n\in\mathbb{Z}^{d}}f\left(
n+x\right)  $ is a periodic function integrable on the torus $\mathbb{R}%
^{d}/\mathbb{Z}^{d}$ and its Fourier coefficients are the restriction of
$\widehat{f}\left(  \xi\right)  $ to the integer points $\mathbb{Z}^{d}$. That
is, formally we have the Poisson summation formula%
\[
\sum_{n\in\mathbb{Z}^{d}}f\left(  n+x\right)  =\sum_{m\in\mathbb{Z}^{d}%
}\widehat{f}\left(  m\right)  e^{2\pi im\cdot x}.
\]
Without additional assumptions on the function $f\left(  x\right)  $, the
series in both sides of this identity do not necessarily converge pointwise.
On the other hand, a positive convergence result can be obtained assuming a
natural regularity condition on the function $f\left(  x\right)  $ and using a
suitable summability method for the Fourier series. The following variation of
the classical Poisson summation formula is especially tailored for our problem.

\begin{theorem}
\label{Thm Poisson}Let $\varphi\left(  x\right)  $ and $f\left(  x\right)  $
be square integrable functions on $\mathbb{R}^{d}$ with compact support.
Assume that $\int_{\mathbb{R}^{d}}\varphi\left(  x\right)  dx=1$ and that, for
every $x$,%
\begin{equation}
f\left(  x\right)  =\lim_{\varepsilon\rightarrow0^{+}}\left\{  \varphi
_{\varepsilon}\ast f\left(  x\right)  \right\}  . \label{Assumption}%
\end{equation}
Then, for every $\varepsilon>0$,
\[
\sum_{m\in\mathbb{Z}^{d}}\left\vert \widehat{\varphi}\left(  \varepsilon
m\right)  \widehat{f}\left(  m\right)  \right\vert <+\infty
\]
and, for every $x$,%
\begin{equation}
\sum_{n\in\mathbb{Z}^{d}}f\left(  n+x\right)  =\lim_{\varepsilon
\rightarrow0^{+}}\left\{  \sum_{m\in\mathbb{Z}^{d}}\widehat{\varphi}\left(
\varepsilon m\right)  \widehat{f}\left(  m\right)  e^{2\pi im\cdot x}\right\}
. \label{Poisson}%
\end{equation}

\end{theorem}

Observe that if $\varphi\left(  x\right)  $ is smooth, then $\widehat{\varphi
}\left(  \xi\right)  $ has fast decay at infinity and the theorem reduces to
the classical Poisson summation formula. See e.g. \cite[Ch. 7, Cor. 2.6]{SW}
and \cite[ Ch. 2, Th. 3.16]{SW} for similar results where $\varphi\left(
x\right)  $ is the Poisson kernel.

\begin{proof}
Under the assumption of the theorem the convolution $\varphi_{\varepsilon}\ast
f\left(  x\right)  $ is continuous with compact support. Then $\sum
_{n\in\mathbb{Z}^{d}}\varphi_{\varepsilon}\ast f\left(  n+x\right)  $ is a
finite sum and gives a continuous function on the torus $\mathbb{R}%
^{d}/\mathbb{Z}^{d}=\left[  0,1\right)  ^{d}$, with Fourier coefficients%
\begin{align*}
&  \int_{\mathbb{R}^{d}/\mathbb{Z}^{d}}\left(  \sum_{n\in\mathbb{Z}^{d}%
}\varphi_{\varepsilon}\ast f\left(  n+x\right)  \right)  e^{-2\pi im\cdot
x}dx\\
&  =\int_{%
{\textstyle\bigcup\nolimits_{n\in\mathbb{Z}^{d}}}
\left\{  n+\left[  0,1\right)  ^{d}\right\}  }\varphi_{\varepsilon}\ast
f\left(  y\right)  e^{-2\pi im\cdot\left(  y-n\right)  }dy\\
&  =\int_{\mathbb{R}^{d}}\varphi_{\varepsilon}\ast f\left(  y\right)  e^{-2\pi
im\cdot y}dy=\widehat{\varphi}\left(  \varepsilon m\right)  \widehat{f}\left(
m\right)  .
\end{align*}
Hence $\sum_{n\in\mathbb{Z}^{d}}\varphi_{\varepsilon}\ast f\left(  n+x\right)
$ has Fourier expansion $\sum_{m\in\mathbb{Z}^{d}}\widehat{\varphi}\left(
\varepsilon m\right)  \widehat{f}\left(  m\right)  e^{2\pi im\cdot x}$. Under
the assumption of the theorem we have, for every $x$,
\[
\lim_{\varepsilon\rightarrow0^{+}}\left\{  \sum_{n\in\mathbb{Z}^{d}}%
\varphi_{\varepsilon}\ast f\left(  n+x\right)  \right\}  =\sum_{n\in
\mathbb{Z}^{d}}f\left(  n+x\right)
\]
Then it is enough to show that for every $x$,%
\begin{equation}
\sum_{n\in\mathbb{Z}^{d}}\varphi_{\varepsilon}\ast f\left(  n+x\right)
=\sum_{m\in\mathbb{Z}^{d}}\widehat{\varphi}\left(  \varepsilon m\right)
\widehat{f}\left(  m\right)  e^{2\pi im\cdot x}. \label{Identity}%
\end{equation}
This follows from the fact that $\sum_{m\in\mathbb{Z}^{d}}\left\vert
\widehat{\varphi}\left(  \varepsilon m\right)  \widehat{f}\left(  m\right)
\right\vert $ converges, which is a consequence of the following
Plancherel-Polya type inequality. Let $g\left(  x\right)  $ be an integrable
function with compact support and let $\psi\left(  x\right)  $ be a smooth
compactly supported function with $\psi\left(  x\right)  =1$ on the support of
$g\left(  x\right)  $. Since $g\left(  x\right)  =\psi\left(  x\right)
g\left(  x\right)  $, $\widehat{g}\left(  \xi\right)  =\widehat{\psi}%
\ast\widehat{g}\left(  \xi\right)  $, and $\widehat{\psi}\left(  \xi\right)  $
is rapidly decreasing, we have%
\begin{align*}
\sum_{m\in\mathbb{Z}^{d}}\left\vert \widehat{g}\left(  m\right)  \right\vert
&  =\sum_{m\in\mathbb{Z}^{d}}\left\vert \int_{\mathbb{R}^{d}}\widehat{\psi
}\left(  m-\xi\right)  \widehat{g}\left(  \xi\right)  \right\vert d\xi\\
&  \leqslant\int_{\mathbb{R}^{d}}\left\{  \sum_{m\in\mathbb{Z}^{d}}\left\vert
\widehat{\psi}\left(  m-\xi\right)  \right\vert \right\}  \left\vert
\widehat{g}\left(  \xi\right)  \right\vert d\xi\\
&  \leqslant\sup_{\xi\in\mathbb{R}^{d}}\left\{  \sum_{m\in\mathbb{Z}^{d}%
}\left\vert \widehat{\psi}\left(  m-\xi\right)  \right\vert \right\}
\,\int_{\mathbb{R}^{d}}\left\vert \widehat{g}\left(  \xi\right)  \right\vert
d\xi\,\\
&  =c\int_{\mathbb{R}^{d}}\left\vert \widehat{g}\left(  \xi\right)
\right\vert d\xi.
\end{align*}
The above the constant $c$ depends on $\psi\left(  x\right)  $, hence on the
support of $g\left(  x\right)  $. See e.g. \cite[Chapter 3]{Nik} for more
general inequalities of this type.

\noindent Applying this inequality to the function $g\left(  x\right)
=\varphi_{\varepsilon}\ast f\left(  x\right)  $ we obtain%
\begin{align*}
\sum_{m\in\mathbb{Z}^{d}}\left\vert \widehat{\varphi_{\varepsilon}\ast
f}\left(  m\right)  \right\vert  &  =\sum_{m\in\mathbb{Z}^{d}}\left\vert
\widehat{\varphi}\left(  \varepsilon m\right)  \widehat{f}\left(  m\right)
\right\vert \leqslant c\int_{\mathbb{R}^{d}}\left\vert \widehat{\varphi
}\left(  \varepsilon\xi\right)  \widehat{f}\left(  \xi\right)  \right\vert
d\xi\\
&  \leqslant c\left(  \int_{\mathbb{R}^{d}}\left\vert \widehat{\varphi}\left(
\varepsilon\xi\right)  \right\vert ^{2}d\xi\right)  ^{1/2}\left(
\int_{\mathbb{R}^{d}}\left\vert \widehat{f}\left(  \xi\right)  \right\vert
^{2}d\xi\right)  ^{1/2}\\
&  =c\varepsilon^{-d/2}\left(  \int_{\mathbb{R}^{d}}\left\vert \varphi\left(
x\right)  \right\vert ^{2}dx\right)  ^{1/2}\left(  \int_{\mathbb{R}^{d}%
}\left\vert f\left(  x\right)  \right\vert ^{2}dx\right)  ^{1/2}.
\end{align*}
Observe that the factor $\varepsilon^{-d/2}$ does not contradict the existence
of the limit as $\varepsilon\rightarrow0^{+}$. The above estimate is just what
we need to show the pointwise equality (\ref{Identity}) for every fixed
$\varepsilon>0$, since we already observed that the limit of the LHS of
(\ref{Identity}) exists.
\end{proof}

Our proof of Pick's theorem below is a corollary of the above version of the
Poisson summation formula, applied to characteristic functions of integer
polygons. Such characteristic functions do not satisfy the assumptions
(\ref{Assumption}) of Theorem \ref{Thm Poisson}, but they can be regularized
by modifying the values at the boundary. It is a classical argument to restate
Pick's theorem in terms of normalized angles as follows. Define a
regularization of the characteristic function of the polygon $P$:
\[
\widetilde{\chi}_{P}\left(  x\right)  =\left\{
\begin{array}
[c]{lll}%
0 &  & \text{if }x\notin P\text{,}\\
1 &  & \text{if }x\ \text{is in the interior of }P\text{,}\\
1/2 &  & \text{if }x\ \text{is in the interior of a side of }P\text{,}\\
\alpha/2\pi &  & \text{if }x\ \text{is a vertex of }P\text{, with interior
angle }\alpha\text{.}%
\end{array}
\right.
\]
Assuming that $P\ $has $N$ vertices, since the sum of the inner angles is
$\pi\left(  N-2\right)  $, we have%
\begin{align*}
\sum_{k\in\mathbb{Z}^{2}}\widetilde{\chi}_{P}\left(  k\right)   &
=\sum_{\text{interior\ points\ of}\ P}1+\sum_{\text{interior\ points\ of
sides\ of}\ P}1/2+\sum_{\text{vertices\ of}\ P}\alpha/2\pi\\
&  =I+\frac{1}{2}(B-N)+\frac{1}{2}(N-2)=I+\frac{1}{2}B-1\ .
\end{align*}

\noindent Hence Pick's theorem is reduced to the following

\begin{theorem}
\label{solid.angle.pick} If $P$ is an integer polygon, then%
\[
\sum_{n\in\mathbb{Z}^{2}}\widetilde{\chi}_{P}\left(  n\right)  =\left\vert
P\right\vert \ .
\]

\end{theorem}

\begin{proof}
We take $\varphi\left(  x\right)  $ radial with compact support and integral
$1$, for example we can take $\varphi\left(  x\right)  =4\pi^{-1}%
\chi_{\left\{  \left\vert x\right\vert <1/2\right\}  }\left(  x\right)  $. For
$\varepsilon>0$ small enough and every $n\in\mathbb{Z}^{2}$ it can be easily
shown that%
\[
\varphi_{\varepsilon}\ast\widetilde{\chi}_{P}\left(  n\right)  =\widetilde
{\chi}_{P}\left(  n\right)  .
\]
Then $\widetilde{\chi}_{P}\left(  x\right)  $ satisfies the assumption
(\ref{Assumption}) in Theorem \ref{Thm Poisson}. Then%
\begin{equation}
\sum_{n\in\mathbb{Z}^{2}}\widetilde{\chi}_{P}\left(  n\right)  =\lim
_{\varepsilon\rightarrow0^{+}}\left\{  \sum_{m\in\mathbb{Z}^{2}}%
\widehat{\varphi}\left(  \varepsilon m\right)  \widehat{\chi}_{P}\left(
m\right)  \right\}  \ . \label{FrequencySum}%
\end{equation}
Observe that in the above identity the limit can be omitted if $\varepsilon$
is small enough. Let $P$ have vertices $\left\{  P_{j}\right\}  \ $and sides
$\left\{  P_{j}+t\left(  P_{j+1}-P_{j}\right)  :t\in\left[  0,1\right]
\right\}  $ with outward unit normals $\left\{  n_{j}\right\}  $. Then the
divergence theorem yields
\begin{align*}
\widehat{\chi}_{P}\left(  m\right)   &  =\int_{P}e^{-2\pi im\cdot x}%
dx=\int_{P}\mathrm{\operatorname{div}}\left(  \frac{-m}{2\pi i\left\vert
m\right\vert ^{2}}e^{-2\pi im\cdot x}\right)  dx\\
&  =\frac{-1}{2\pi i}\sum_{j=1}^{N}\frac{m\cdot n_{j}}{\left\vert m\right\vert
^{2}}\left\vert P_{j+1}-P_{j}\right\vert \int_{0}^{1}e^{-2\pi im\cdot\left(
P_{j}+t\left(  P_{j+1}-P_{j}\right)  \right)  }dt\\
&  =\frac{-1}{2\pi i}\sum_{j=1}^{N}\frac{m\cdot n_{j}}{\left\vert m\right\vert
^{2}}\left\vert P_{j+1}-P_{j}\right\vert \ e^{-\pi im\cdot\left(
P_{j+1}+P_{j}\right)  }\frac{\sin\left(  \pi m\cdot\left(  P_{j+1}%
-P_{j}\right)  \right)  }{\pi m\cdot\left(  P_{j+1}-P_{j}\right)  }%
\end{align*}
(with $P_{N+1}=P_{1}$ and obvious modifications when $m\cdot\left(
P_{j+1}-P_{j}\right)  =0$). When $P_{j}$ and $m$ belong to $\mathbb{Z}^{2}$,
then%
\begin{align*}
&  e^{-\pi im\cdot\left(  P_{j+1}+P_{j}\right)  }\frac{\sin\left(  \pi
m\cdot\left(  P_{j+1}-P_{j}\right)  \right)  }{\pi m\cdot\left(  P_{j+1}%
-P_{j}\right)  }\\
&  =e^{-2\pi im\cdot P_{j}}e^{-\pi im\cdot\left(  P_{j+1}-P_{j}\right)  }%
\frac{\sin\left(  \pi m\cdot\left(  P_{j+1}-P_{j}\right)  \right)  }{\pi
m\cdot\left(  P_{j+1}-P_{j}\right)  }\\
&  =\left\{
\begin{array}
[c]{cc}%
0 & \text{if\ \ }m\cdot\left(  P_{j+1}-P_{j}\right)  \neq0,\\
1 & \text{if\ \ }m\cdot\left(  P_{j+1}-P_{j}\right)  =0.
\end{array}
\right.
\end{align*}
Recalling that $\widehat{\varphi}\left(  0\right)  =1$ and $\widehat{\chi}%
_{P}\left(  0\right)  =\left\vert P\right\vert $, and that $\widehat{\varphi
}\left(  m/R\right)  $ is radial, hence even, while $m\cdot n_{j}$ is odd, we
obtain
\begin{align}
&  \sum_{m\in\mathbb{Z}^{2}}\widehat{\varphi}\left(  \varepsilon m\right)
\widehat{\chi}_{P}\left(  m\right)  =\widehat{\chi}_{P}\left(  0\right)
+\sum_{m\in\mathbb{Z}^{2}\setminus\left\{  0\right\}  }\widehat{\varphi
}\left(  \varepsilon m\right)  \widehat{\chi}_{P}\left(  m\right)
\label{magic}\\
&  =\left\vert P\right\vert -\frac{1}{2\pi i}\sum_{j=1}^{N}\left\vert
P_{j+1}+P_{j}\right\vert \left(  \sum_{m\neq0,\ m\cdot\left(  P_{j+1}%
-P_{j}\right)  =0}\widehat{\varphi}\left(  m/R\right)  \frac{m\cdot n_{j}%
}{\left\vert m\right\vert ^{2}}\right)  =\left\vert P\right\vert \ .\nonumber
\end{align}

\end{proof}

\bigskip

Let us conclude with some remarks and a conjecture.

Pick's theorem, in the naive form that we know it, fails in dimension $d\geq
3$. Indeed, as observed by J.E. Reeve, the tetrahedron with vertices $\left(
0,0,0\right)  $, $\left(  1,0,0\right)  $, $\left(  0,1,0\right)  $, $\left(
1,1,N\right)  $, has volume $N/6$, contains four integer points on the
boundary, and has no integer points inside. Hence there is no simple relation
between the volume and the integer points for general 3-dimensional integer polytopes.

Fascinating relations do appear, however, when an integer polyhedron is
dilated by an integer factor. By Ehrhart's theorem from the 1950's, the number
of integer points in a dilated integer polyhedron $P$ is a polynomial function
of the integer dilation parameter, with leading coefficient equal to the
volume of $P$. The reader may consult, for example, the books \cite{Barvinok}
and \cite{BR}.

The above defined \textit{regularized discrete volume} $\sum_{n\in
\mathbb{Z}^{d}}\widetilde{\chi}_{P}\left(  n\right)  $ can be easily defined
in every dimension, but in general it is no longer equal to the Euclidean
volume $\left\vert P\right\vert $. However, as we see from equations
\ref{FrequencySum} and \ref{magic}, it is still true that
\[
\sum_{n\in\mathbb{Z}^{d}}\widetilde{\chi}_{P}\left(  n\right)  =\left\vert
P\right\vert \text{\ \ \ \ \ if and only if\ \ \ \ \ }\sum_{0\neq
m\in\mathbb{Z}^{d}}\widehat{\varphi}\left(  \varepsilon m\right)
\widehat{\chi}_{P}\left(  m\right)  =0
\]
for all sufficiently small $\varepsilon>0$, and for every choice of
$\varphi\left(  x\right)  $ as before.

If an integer polytope $P$ satisfies%
\begin{equation}
\sum_{n\in\mathbb{Z}^{d}}\widetilde{\chi}_{P}\left(  n\right)  =\left\vert
P\right\vert , \label{Higher1}%
\end{equation}
that is if its continuous Euclidean volume is equal to its regularized
discrete volume, then we call such an integer polytope a \textit{concrete
polytope}. Here we follow the tradition of \cite{GrahamKnuthPatachnik} in
using the first three letters of `continuous' and the last 5 letters of
`discrete' to consider objects that can be described by both continuous
methods and by discrete methods.

An interesting open problem is to characterize the concrete integer polytopes
in $\mathbb{R}^{d}$; that is, what are the integer polytopes which enjoy the
relation $\sum_{n\in\mathbb{Z}^{d}}\widetilde{\chi}_{P}\left(  n\right)
=\left\vert P\right\vert \,$? In other words, this class of integer polytopes
gives a natural extension to higher dimensions for the Pick-type property
(Theorem \ref{solid.angle.pick}) that we saw in dimension $2$.

As already shown by Barvinok \cite{Barvinok}, integer zonotopes (polytopes
whose faces, of all dimensions, are symmetric) are concrete polytopes. A more
general family of concrete polytopes is given by polytopes that multi-tile
Euclidean space. Given an integer $k$, a polytope $P\subset\mathbb{R}^{d}$
multi-tiles (or $k$-tiles) $\mathbb{R}^{d}$ with a discrete set of translation
vectors $\mathcal{L}$ if each point $x\in\mathbb{R}^{d}$ is covered $k$-times
(except for $\partial P$ and its translates) by the translations of $P$, from
the set of translation vectors $\mathcal{L}$.

Indeed, a periodization argument, used extensively by Kolountzakis (see
\cite{GRS}, \cite[p.137]{Kol} ), tells us that the integer polytope $P$
multi-tiles $\mathbb{R}^{d}$ with the lattice $\mathbb{Z}^{d}$ of integer
translations, if and only if $\widehat{\chi}_{P}\left(  m\right)  =0$ for
every $m\in\mathbb{Z}^{d}\mathbb{\setminus}\{0\}$. So we see that identity
(\ref{Higher1}) is trivially satisfied in this case, and therefore the
$k$-tiling integer polytopes are concrete polytopes.

The following result \cite{GRS} gives a characterization of the integer
polytopes that multi-tile under translations.

\begin{theorem}
\label{multi-tiling}Let $P$ be an integer polytope in $\mathbb{R}^{d}$. Then
$P$ multi-tiles $\mathbb{R}^{d}$ if and only if $P$ is a symmetric polytope,
and all of its facets (codimension-1 faces) are symmetric polytopes.
\end{theorem}

We therefore see that any integer polytope $P$ which is symmetric, and has
symmetric facets, is a concrete polytope. Are there other, more general
classes of concrete polytopes? Yes! In dimension two any integer triangle is
concrete. Even in all higher dimensions the answer is affirmative. For
example, consider the tetrahedron $\Delta$ defined as the convex hull of the
vectors $(0,0,0),(1,0,0),(1,1,0)$, and $(1,1,1)$. With six reflections of
$\Delta$ about its facets, we can reconstruct the unit cube. Since these
reflections are isometries and preserve $\mathbb{Z}^{d}$, both the discrete
and continuous volumes of $\Delta$ are equal to $1/6$. Hence $\Delta$ is
concrete, but lies outside the class of multi-tiling polytopes.

We invite the reader to reflect upon the following conjectured extension of
Pick's theorem to higher dimensional objects.

\begin{conjecture}
\label{conjecture} Suppose that an integer polytope $P$ is a concrete
polytope. Then $P$ multi-tiles $\mathbb{R}^{d}$ by translations together with
a finite set of reflections.
\end{conjecture}

\end{document}